\documentclass{amsart}
\usepackage{amssymb}
\usepackage{amsthm}
\theoremstyle{plain}
\newtheorem{theor}{Theorem} \newtheorem{lem}{Lemma} \newtheorem{cor}{Corollary}
\newtheorem{prop}{Proposition}
\theoremstyle{remark} \newtheorem{rem}{Remark} \theoremstyle{definition}
\newtheorem{defin}{Definition} \newtheorem{ex}{Example}
\begin{document}
\title{Chebyshev polynomials, Zolotarev polynomials and plane trees}
\author{Yury Kochetkov}
\date{19.12.2012}
\email{yuyk@prov.ru}  \maketitle
\begin{abstract}{A polynomial with exactly two
critical values is called a generalized Chebyshev polynomial. A polynomial with
exactly three critical values is called a Zolotarev polynomial. Two Chebyshev
polynomials $f$ and $g$ are called Z-homotopic, if there exists a family
$p_\alpha$, $\alpha\in [0,1]$, where $p_0=f$, $p_1=g$ and $p_\alpha$ is a
Zolotarev polynomial, if $\alpha\in (0,1)$. As each Chebyshev polynomial
defines a plane tree (and vice versa), Z-homotopy can be defined for plane
trees. In this work we prove some necessary geometric conditions for plane
trees Z-homotopy, describe Z-homotopy for trees with 5 and 6 edges and study
one interesting example in the class of trees with 7 edges.}\end{abstract}

\section{Introduction}

\subsection{Generalized Chebyshev polynomials} Polynomial $p(z)\in\mathbb C[z]$
is called a generalized Chebyshev polynomial if it has exactly two finite
critical values --- $\alpha$ and $\beta$ (in what follows we will call such
polynomial simply a Chebyshev polynomial). If $p(z)$ is a Chebyshev polynomial,
then the set $p^{-1}[\alpha,\beta]$ is a plane connected tree $T_p$ (see,
\cite{1}, for example). Inverse images of points $\alpha$ and $\beta$ are
vertices of tree $T_p$ and the degree of a vertex equals to the multiplicity of
the corresponding critical point (a vertex of degree 1 is a simple root of
polynomial $p(z)-\alpha$ or $p(z)-\beta$). Also for each plane tree $T$ there
exists a Chebyshev polynomial $p(z)$, defined up to linear change of variable
$z$ and variable $u=p(z)$, such that trees $p^{-1}[\alpha,\beta]$ and $T$ are
isotopic. Such polynomial $p(z)$ will be called a \emph{polynomial that defines
the tree} $T$.

Vertices of a plane tree $T$ can be painted in two colors
--- black and white so, that colors of any two adjacent vertices are different.
Such painting will be called a \emph{binary structure} of $T$. Obviously,
vertices of one color are inverse images of $\alpha$ and vertices of another
another color --- of $\beta$.

The type (or passport) of plane tree with binary structure is two sequences of
multiplicities of white vertices and black vertices, respectively, in
nonincreasing order. Thus the type of the tree
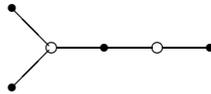
\[\begin{picture}(75,40) \put(0,5){\circle*{3}} \put(0,35){\circle*{3}}
\put(15,20){\circle{4}} \put(35,20){\circle*{3}} \put(55,20){\circle{4}}
\put(75,20){\circle*{3}} \put(0,5){\line(1,1){14}} \put(0,35){\line(1,-1){14}}
\put(17,20){\line(1,0){36}} \put(57,20){\line(1,0){18}} \end{picture}\] is
$\langle 3,2\,|\,2,1,1,1\rangle$.

\begin{rem} Often it is assumed that numbers $\alpha$ and $\beta$ are 0 and 1.
\end{rem}

\subsection{Zolotarev polynomials} A polynomial $p\in\mathbb{C}[z]$ is called a
\emph{Zolotarev polynomial} if it has exactly three finite critical values. If
$p$ is a Zolotarev polynomial, ${\rm deg}(p)=n$, $\alpha$, $\beta$ and $\gamma$
its critical values and $C$ is a simple arc $C\subset\mathbb{C}$, that connects
points $\alpha$, $\beta$ and $\gamma$, then $p^{-1}(C)$ is a connected plane
tree with $2n$ edges. Here points from the set $p^{-1}\{\alpha,\beta,\gamma\}$
are vertices of this tree and degree of a vertex $v$, $p(v)=\alpha$, equals to
multiplicity of critical point $v$, if $\alpha$ is an endpoint of $C$, or to
double multiplicity, if $\alpha$ is an interior point. Vertices of the tree
$p^{-1}(C)$ can be painted in three colors: white, black and grey, where white
vertices are inverse images of the interior (with respect to arc $C$) critical
value. One vertex of each edge is white and other --- black or grey.

\begin{rem} Arcs $C_1$ and $C_2$, that connect points $\alpha$,
$\beta$ and $\gamma$, can be isotopically nonequivalent
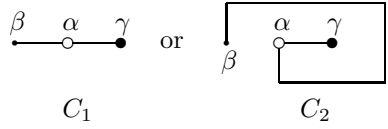
\[\begin{picture}(140,50) \put(0,30){\circle*{2}} \put(20,30){\circle{4}}
\put(40,30){\circle*{4}} \put(0,30){\line(1,0){18}} \put(22,30){\line(1,0){18}}
\put(18,2){$C_1$} \put(55,28){or} \put(80,30){\circle*{2}}
\put(100,30){\circle{4}} \put(120,30){\circle*{4}} \put(80,30){\line(0,1){15}}
\put(80,45){\line(1,0){60}} \put(140,45){\line(0,-1){30}}
\put(140,15){\line(-1,0){40}} \put(100,15){\line(0,1){13}}
\put(102,30){\line(1,0){18}} \put(108,2){$C_2$} \put(-2,35){$\beta$}
\put(18,35){$\alpha$} \put(38,35){$\gamma$} \put(78,20){$\beta$}
\put(98,35){$\alpha$} \put(118,35){$\gamma$} \end{picture}\] for example. In
this case trees $p^{-1}(C_1)$ and $p^{-1}(C_2)$ also can be isotopically
nonequivalent.
\end{rem}
The \emph{passport} of Zolotarev polynomial is three sequences of
multiplicities of its critical points that correspond to the first, the second
and the third critical value, respectively. Multiplicity sequences will be
written in the nonincreasing order $\langle
k_1,k_2,\ldots\,|\,l_1,\l_2,\ldots\,|\,m_1,m_2,\ldots\rangle$. Critical points
of polynomial $p=x^2(x-1)^2(3x-1)$, for example, are $0,1,2/3$ and $1/5$ with
values $0,0,4/81$ and $-32/3125$, respectively. So $\langle
2,2\,|\,2\,|\,2\rangle$ is the passport of $p$.

\section{Z-homotopy}

\begin{defin} Two trees $T_1$ and $T_2$ will be called Z-homotopic if
there exists a continuous family $p_\lambda\in\mathbb{C}[z]$, $\lambda\in
[0,1]$, such that
\begin{itemize}
    \item all polynomials $p_\lambda$ has the same degree;
    \item polynomial $p_0$ is a Chebyshev polynomial and defines the tree
    $T_1$;
    \item polynomial $p_1$ is a Chebyshev polynomial and defines the tree
    $T_2$;
    \item polynomials $p_\lambda$, $\lambda\neq 0,1$, are Zolotarev
    polynomials, but \emph{not} Chebyshev polynomials.
\end{itemize}
\end{defin}

\begin{ex} Let us study the Z-homotopy problem on the set of 5-edge
trees. There are five of them:
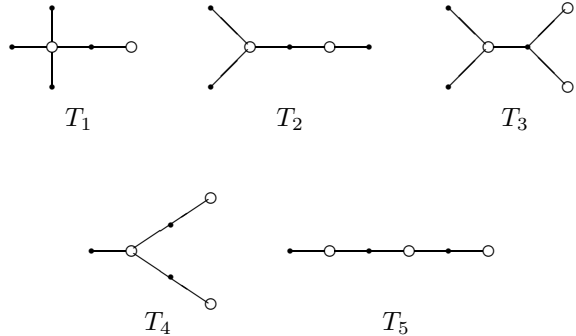
\[\begin{picture}(210,60) \put(0,35){\circle*{2}} \put(15,35){\circle{4}}
\put(15,20){\circle*{2}} \put(15,50){\circle*{2}} \put(30,35){\circle*{2}}
\put(45,35){\circle{4}} \put(0,35){\line(1,0){13}}
\put(17,35){\line(1,0){26}}\put(15,20){\line(0,1){13}}
\put(15,37){\line(0,1){13}} \put(20,5){$T_1$}

\put(75,20){\circle*{2}} \put(75,50){\circle*{2}} \put(90,35){\circle{4}}
\put(105,35){\circle*{2}} \put(120,35){\circle{4}} \put(135,35){\circle*{2}}
\put(75,20){\line(1,1){14}} \put(75,50){\line(1,-1){14}}
\put(92,35){\line(1,0){26}} \put(122,35){\line(1,0){13}} \put(100,5){$T_2$}

\put(165,20){\circle*{2}} \put(165,50){\circle*{2}} \put(180,35){\circle{4}}
\put(195,35){\circle*{2}} \put(210,20){\circle{4}} \put(210,50){\circle{4}}
\put(165,20){\line(1,1){14}} \put(165,50){\line(1,-1){14}}
\put(182,35){\line(1,0){13}} \put(195,35){\line(1,1){14}}
\put(195,35){\line(1,-1){14}} \put(185,5){$T_3$}\end{picture}\]

\[\begin{picture}(150,60)
\put(0,35){\circle*{2}} \put(15,35){\circle{4}} \put(30,25){\circle*{2}}
\put(30,45){\circle*{2}} \put(45,15){\circle{4}} \put(45,55){\circle{4}}
\put(0,35){\line(1,0){13}} \put(16,36){\line(3,2){28}}
\put(16,34){\line(3,-2){28}} \put(20,5){$T_4$}

\put(75,35){\circle*{2}} \put(90,35){\circle{4}} \put(105,35){\circle*{2}}
\put(120,35){\circle{4}} \put(135,35){\circle*{2}} \put(150,35){\circle{4}}
\put(75,35){\line(1,0){13}} \put(75,35){\line(1,0){13}}
\put(92,35){\line(1,0){26}} \put(122,35){\line(1,0){26}} \put(110,5){$T_5$}
\end{picture}\]

Let $p=\int x^2(x-1)(x-a)\,dx$. Critical points of $p$ are $0,1$ and $a$ and
$0$, $5a-3$ and $a^4(5-3a)$ are corresponding critical values. If
\begin{itemize}
    \item $a=0$, then $p$ is a Chebyshev polynomial that defines the tree $T_1$;
    \item $a=1$, then $p$ is a Chebyshev polynomial that defines the tree $T_3$;
    \item $a=3/5$, then $p(1)=0$ and $p$ is a Chebyshev polynomial that defines
    the tree $T_2$;
    \item $a=5/3$, then $p(a)=0$ and $p$ is a Chebyshev polynomial that defines
    the tree $T_2$;
    \item $a=(-2\pm\sqrt{5}\,i)/3$, then $p(a)=p(1)$ and $p$ is a Chebyshev
    polynomial that defines the tree $T_4$.
\end{itemize}
For all other values of parameter $a$ the polynomial $p$ is Zolotarev
polynomial. Thus deformations of parameter $a$ allows one to realize pairwise
Z-homotopies between trees $T_1$, $T_2$, $T_3$ and $T_4$. For example the
following deformation of tree corresponds to the increase of parameter $a$ from
$0$ to $3/5$ (arc $C$ in this case is the segment, that connects critical
values $5a-3$ and $a^3(5-3a)$):

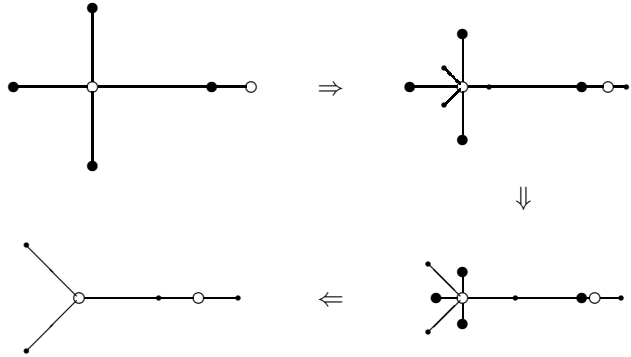
\[\begin{picture}(232,150) \put(0,115){\circle*{4}} \put(30,115){\circle{4}}
\put(30,145){\circle*{4}} \put(30,85){\circle*{4}} \put(75,115){\circle*{4}}
\put(90,115){\circle{4}} \put(0,115){\line(1,0){28}}
\put(30,117){\line(0,1){28}} \put(30,85){\line(0,1){28}}
\put(32,115){\line(1,0){56}} \put(115,112){$\Rightarrow$}

\put(150,115){\circle*{4}} \put(170,115){\circle{4}} \put(170,135){\circle*{4}}
\put(170,95){\circle*{4}} \put(215,115){\circle*{4}} \put(225,115){\circle{4}}
\put(232,115){\circle*{2}} \put(180,115){\circle*{2}}
\put(163,122){\circle*{2}} \put(163,108){\circle*{2}}
\put(150,115){\line(1,0){18}} \put(170,117){\line(0,1){18}}
\put(170,95){\line(0,1){18}} \put(172,115){\line(1,0){51}}
\qbezier(163,122)(166,119)(169,116) \qbezier(163,108)(166,111)(169,114)
\put(227,115){\line(1,0){5}} \put(190,70){$\Downarrow$}

\put(160,35){\circle*{4}} \put(170,35){\circle{4}} \put(170,45){\circle*{4}}
\put(170,25){\circle*{4}} \put(215,35){\circle*{4}} \put(220,35){\circle{4}}
\put(230,35){\circle*{2}} \put(190,35){\circle*{2}} \put(157,48){\circle*{2}}
\put(157,22){\circle*{2}} \put(160,35){\line(1,0){8}}
\put(170,37){\line(0,1){8}} \put(170,25){\line(0,1){8}}
\put(172,35){\line(1,0){46}} \put(222,35){\line(1,0){8}}
\put(157,48){\line(1,-1){12}} \put(157,22){\line(1,1){12}}
\put(115,32){$\Leftarrow$}

\put(5,55){\circle*{2}} \put(5,15){\circle*{2}} \put(25,35){\circle{4}}
\put(55,35){\circle*{2}} \put(70,35){\circle{4}} \put(85,35){\circle*{2}}
\put(5,55){\line(1,-1){19}} \put(5,15){\line(1,1){19}}
\put(27,35){\line(1,0){41}} \put(72,35){\line(1,0){13}}
\end{picture}\]

Trees  $T_1$, $T_2$ and $T_4$ are Z-homotopic to tree $T_5$. Indeed, let us
consider the polynomial $p(x)=\int x(x-1)(x-a)(x-b)\,dx$. If $p(a)=p(0)$,
$a\neq 2$, then this polynomial is a Zolotarev polynomial (here
$b=(3a^2-5a)/(5a-10)$). However, for some values of parameter $a$ polynomial
$p$ degenerates into Chebyshev polynomial. Indeed,
\begin{enumerate}
    \item  if $a=0$, then $b=0$, and we have a Chebyshev polynomial, that
    defines the tree $T_1$;
    \item if $a=1$, then $b=2/5$ and $p(1)=0$, and we have a Chebyshev
    polynomial, that defines the tree $T_2$;
    \item if $a=5/3$, then $b=0$, and we have a Chebyshev
    polynomial, that defines the tree $T_2$;
    \item if $a=\pm\sqrt 5$, then $b=1\pm\sqrt 5$ and $p(1)=p(b)$, and we
    have a Chebyshev polynomial, that defines the tree $T_5$;
    \item if $a=(5\pm\sqrt 5)/4$, then $b=-(1\pm\sqrt 5)/4$ and $p(1)=p(b)$,
    and we have a Chebyshev polynomial, that defines the tree $T_5$;
    \item if $a=(5\pm\sqrt 5\,i)/3$, then $b=1$, and we have a Chebyshev
    polynomial, that defines the tree $T_4$.
\end{enumerate}
Thus, a deformation of parameter $a$ allows us to construct a Z-homotopy
between trees $T_1$ and $T_5$, $T_2$ and $T_5$, $T_4$ and $T_5$.

Trees $T_3$ and $T_5$ are not Z-homotopic. This statement will be proved in the
next section. Also it is a consequence of results in section "Theorems".
\end{ex}

\section{Geometry of space of Zolotarev polynomials of degree 5}

Let $q=x^4+ax^2+bx+c$ and $p=\int q\,dx$. The polynomial $p$ is a Zolotarev
polynomial if among numbers $p(x_1),p(x_2),p(x_3),p(x_4)$, where
$x_1,x_2,x_3,x_4$ are roots of $q$, there are only three different. In this
case the polynomial $s(y)=(y-p(x_1))(y-p(x_2))(y-p(x_3))(y-p(x_4))$ has a
multiple root., i.e. its discriminant is zero. This discriminant is reducible:
\begin{multline*}
(1280a^6-32256a^4c+9504a^3b^2+269568a^2c^2-69984ab^2c-19683b^4-746496c^3)
\times\\ (16a^4c-4a^3b^2-128a^2c^2+144ab^2c-27b^4+256c^3)=0.\end{multline*} We
see that the variety of Zolotarev polynomials of degree 5 is reducible and has
two components $C_1$ and $C_2$. The second factor, that defines the component
$C_2$, is simply the discriminant of polynomial $q$.

Intersection $C_1\cap C_2$ is the union of 3 components.
\begin{itemize}
    \item Polynomials that belong to the first component are Chebyshev polynomials that
          define the tree $T_4$.
    \item Polynomials that belong to the second component are Chebyshev polynomials
          that define the tree $T_2$.
    \item Polynomials that belong to the third component are Chebyshev polynomials
          that define the tree $T_1$.
\end{itemize}

A Chebyshev polynomial $p_0$ that defines $T_5$ belongs only to the first
component $C_1$ and a Chebyshev polynomial $p_1$ that defines $T_3$ belongs
only to the second component $C_2$. Thus a family of Zolotarev polynomials
which connect $p_0$ and $p_1$ must also contain one of Chebychev polynomials in
$C_1\cap C_2$. But then this family is not a Z-homotopy.

\section{Theorems}

In this section we will prove a necessary condition for Z-homotopy existence
(i.e. a sufficient condition for its absence).

\begin{lem}  Let $p_\lambda$, $0<\lambda<1$ be a continuous family of
Zolotarev polynomials of degree $n$. Then passports of all these polynomials
are the same.\end{lem}

\begin{proof} Let $a_\lambda$, $b_\lambda$ ¨ $c_\lambda$ be critical values of
polynomial $p_\lambda$. They are continuous functions of parameter $\lambda$. A
change of passport during increase or decrease of parameter $\lambda$ can occur
only in the case of collision of roots of polynomial $p_\lambda-a_\lambda$ (or
$p_\lambda-b_\lambda$, or $p_\lambda-c_\lambda$): two roots $x'_\lambda$ and
$x''_\lambda$ of polynomial $p_\lambda-a_\lambda$ of multiplicities $k'$ and
$k''$, respectively, approach to each other, when $\lambda\rightarrow \mu$, and
generate a root $x_\mu$ of polynomial $p_\mu-a_\mu$ of multiplicity $k'+k''-1$.

Let the passport of $p_\lambda$ be $\langle
k_1,\ldots,k_r\,|\,l_1,\ldots,l_s\,|\,m_1,\ldots,m_t\rangle$. Then
$$\sum_{i=1}^r k_i=n,\quad \sum_{i=1}^s l_i=n,\quad \sum_{i=1}^tm_i=n$$ and
$$\sum_{i=1}^r(k_i-1)+\sum_{i=1}^s(l_i-1)+\sum_{i=1}^t(m_i-1)=n-1.$$ Hence,
$r+s+t=2n+1$. But the collision of roots diminishes the number $r$ and violates
the above equality. \end{proof}

\begin{rem} We see, that it is more correct to speak not about Z-homotopy, but
about Z-homotopy in the class of Zolotarev polynomials with a given passport.
Thus, trees $T_1$, $T_2$, $T_3$ and $T_4$ with 5 edges are pairwise Z-homotopic
in the class of Zolotarev polynomials with the passport $\langle
3\,|\,2\,|\,2\rangle$ and trees $T_1$ and $T_5$, $T_2$ and $T_5$, $T_4$ and
$T_5$ are Z-homotopic in the class of Zolotarev polynomials with the passport
$\langle 2,2\,|\,2\,|\,2\rangle$. \end{rem}

\begin{lem} Let $p_\lambda$, $0\leqslant\lambda<1$, be a continuous
family of polynomials of degree $n$, where $p_0$ is a Chebyshev polynomial and
$p_\lambda$, $\lambda>0$, are Zolotarev polynomials (but not Chebyshev
polynomials). Let us assume that a critical point $a$ of polynomial $p_0$ of
multiplicity $k$ generates $m$, $m>1$, critical points
$a_1(\lambda),\ldots,a_m(\lambda)$ in the family $p_\lambda$ with
multiplicities $k_1,\ldots,k_m$. Then numbers $p_\lambda(a_1(\lambda)),\ldots,
p_\lambda(a_m(\lambda))$ cannot all be equal.
\end{lem}

\begin{proof} Let us assume that the opposite is true:
$$p_\lambda(a_1(\lambda))=\ldots=p_\lambda(a_m(\lambda))=\alpha(\lambda).$$
Let $\lambda\to 0$. Then
$$a_i(\lambda)\to a,\,i=1,\ldots,m,\text{ and } \alpha(\lambda)\to
\alpha=p_0(a).$$ But $k-1=(k_1-1)+\ldots+(k_m-1)$, so $a$ is a root of
polynomial $p_0-\alpha$ of multiplicity $k+m-1$. We have a contradiction.
\end{proof}

\begin{defin} A tree is called a \emph{chain}, if valences of all its
vertices are $\leqslant 2$.\end{defin}

\begin{theor} If a tree $T$ has a white vertex $a$ of degree $\geqslant
3$ and a black vertex $b$ of degree $\geqslant 3$, then it cannot be
Z-homotopic to a chain.\end{theor}

\begin{proof} Let us assume that the opposite is true. Then there exist a
Z-homotopy connecting a Chebyshev polynomial $p_0$, that defines $T$, with a
Chebyshev polynomial $p_1$, that defines the chain. It means that critical
points $a$ and $b$ in the family $p_\lambda$ generated critical points
$a_1,\ldots,a_m$ and $b_1,\ldots,b_n$, respectively, all of them of
multiplicity 2. Let $p_0(a)=\alpha$ and $p_0(b)=\beta$. If parameter $\lambda$
is small, then values $p_\lambda(a_1),\ldots,p_\lambda(a_m)$ are close to
$\alpha$ and among them are at least two different. Analogously, values
$p_\lambda(b_1),\ldots,p_\lambda(b_n)$ are close to $\beta$ and among them are
at least two different. But then a polynomials $p_\lambda$, $\lambda\ll 1$, has
at least 4 critical values. We have a contradiction. \end{proof}

\begin{cor} Trees $T_3$ and $T_5$ cannot be Z-homotopic.\end{cor}

\section{Trees with six edges}

Below are all plane 6-edge trees up to mirror symmetry (the designation of
symmetrical tree is in brackets):
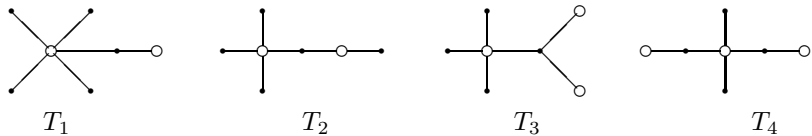
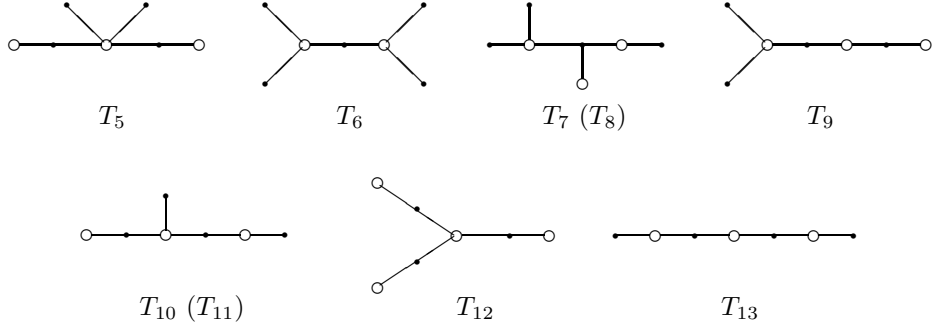
\[\begin{picture}(300,55) \put(0,50){\circle*{2}} \put(0,20){\circle*{2}}
\put(15,35){\circle{4}} \put(30,50){\circle*{2}} \put(30,20){\circle*{2}}
\put(40,35){\circle*{2}} \put(55,35){\circle{4}} \put(0,50){\line(1,-1){14}}
\put(0,20){\line(1,1){14}} \put(16,36){\line(1,1){14}}
\put(16,34){\line(1,-1){14}} \put(17,35){\line(1,0){36}} \put(12,5){$T_1$}

\put(80,35){\circle*{2}} \put(95,35){\circle{4}} \put(95,50){\circle*{2}}
\put(95,20){\circle*{2}} \put(110,35){\circle*{2}} \put(125,35){\circle{4}}
\put(140,35){\circle*{2}} \put(80,35){\line(1,0){13}}
\put(95,37){\line(0,1){13}} \put(95,33){\line(0,-1){13}}
\put(97,35){\line(1,0){26}} \put(127,35){\line(1,0){13}} \put(110,5){$T_2$}

\put(165,35){\circle*{2}} \put(180,35){\circle{4}} \put(180,50){\circle*{2}}
\put(180,20){\circle*{2}} \put(200,35){\circle*{2}} \put(215,50){\circle{4}}
\put(215,20){\circle{4}} \put(165,35){\line(1,0){13}}
\put(180,37){\line(0,1){13}} \put(180,33){\line(0,-1){13}}
\put(182,35){\line(1,0){18}} \put(200,35){\line(1,1){14}}
\put(200,35){\line(1,-1){14}} \put(190,5){$T_3$}

\put(240,35){\circle{4}} \put(255,35){\circle*{2}} \put(270,35){\circle{4}}
\put(270,50){\circle*{2}} \put(270,20){\circle*{2}} \put(285,35){\circle*{2}}
\put(300,35){\circle{4}} \put(242,35){\line(1,0){26}}
\put(270,37){\line(0,1){13}} \put(270,33){\line(0,-1){13}}
\put(272,35){\line(1,0){26}} \put(280,5){$T_4$} \end{picture}\]

\[\begin{picture}(345,55) \put(0,35){\circle{4}} \put(15,35){\circle*{2}}
\put(35,35){\circle{4}} \put(55,35){\circle*{2}} \put(70,35){\circle{4}}
\put(20,50){\circle*{2}} \put(50,50){\circle*{2}} \put(2,35){\line(1,0){31}}
\put(37,35){\line(1,0){31}} \put(20,50){\line(1,-1){14}}
\put(50,50){\line(-1,-1){14}} \put(32,5){$T_5$}

\put(95,50){\circle*{2}} \put(95,20){\circle*{2}} \put(110,35){\circle{4}}
\put(125,35){\circle*{2}} \put(140,35){\circle{4}} \put(155,50){\circle*{2}}
\put(155,20){\circle*{2}} \put(95,20){\line(1,1){14}}
\put(95,50){\line(1,-1){14}} \put(112,35){\line(1,0){26}}
\put(141,36){\line(1,1){14}} \put(141,34){\line(1,-1){14}} \put(122,5){$T_6$}

\put(180,35){\circle*{2}} \put(195,35){\circle{4}} \put(195,50){\circle*{2}}
\put(215,35){\circle*{2}} \put(215,20){\circle{4}} \put(230,35){\circle{4}}
\put(245,35){\circle*{2}} \put(180,35){\line(1,0){13}}
\put(195,37){\line(0,1){13}} \put(197,35){\line(1,0){31}}
\put(215,35){\line(0,-1){13}} \put(232,35){\line(1,0){13}} \put(200,5){$T_7$
($T_8$)}

\put(270,50){\circle*{2}} \put(270,20){\circle*{2}} \put(285,35){\circle{4}}
\put(300,35){\circle*{2}} \put(315,35){\circle{4}} \put(330,35){\circle*{2}}
\put(345,35){\circle{4}} \put(270,50){\line(1,-1){14}}
\put(270,20){\line(1,1){14}} \put(287,35){\line(1,0){26}}
\put(317,35){\line(1,0){26}} \put(300,5){$T_9$}
\end{picture}\]

\[\begin{picture}(290,55) \put(0,35){\circle{4}} \put(15,35){\circle*{2}}
\put(30,35){\circle{4}} \put(30,50){\circle*{2}} \put(45,35){\circle*{2}}
\put(60,35){\circle{4}} \put(75,35){\circle*{2}} \put(2,35){\line(1,0){26}}
\put(30,37){\line(0,1){13}} \put(32,35){\line(1,0){26}}
\put(62,35){\line(1,0){13}} \put(20,5){$T_{10}$ ($T_{11}$)}

\put(110,55){\circle{4}} \put(110,15){\circle{4}} \put(125,45){\circle*{2}}
\put(125,25){\circle*{2}} \put(140,35){\circle{4}} \put(160,35){\circle*{2}}
\put(175,35){\circle{4}} \put(111,54){\line(3,-2){28}}
\put(111,16){\line(3,2){28}} \put(142,35){\line(1,0){31}} \put(140,5){$T_{12}$}

\put(200,35){\circle*{2}} \put(215,35){\circle{4}} \put(230,35){\circle*{2}}
\put(245,35){\circle{4}} \put(260,35){\circle*{2}} \put(275,35){\circle{4}}
\put(290,35){\circle*{2}} \put(200,35){\line(1,0){13}}
\put(217,35){\line(1,0){26}} \put(247,35){\line(1,0){26}}
\put(277,35){\line(1,0){13}} \put(240,5){$T_{13}$} \end{picture}\]

By Theorem 1 from the previous section, trees $T_3$ and $T_{13}$, $T_7$ and
$T_{13}$, $T_8$ and $T_{13}$ are not Z-homotopic. However, there is one more
non-homotopic pair.

\begin{prop} Trees $T_6$ and $T_{12}$ are not Z-homotopic.\end{prop}

\begin{proof} Let the opposite be true and let $a$ and $b$ be white vertices of
degree 3 of the tree $T_6$.
\par\noindent
\underline{The first case.} Let polynomials $p_\lambda$ have a critical point
$a_\lambda$ of multiplicity 3, all other critical points are of multiplicity 2.
Thus the vertex $b$ generates two critical points $b_1$ and $b_2$ of
multiplicity 2, $p_\lambda(b_1)\neq p_\lambda(b_2)$ and value $p_\lambda(a)$
coincides with value $p_\lambda(b_1)$ or with values $p_\lambda(b_2)$. But then
tree $T_{12}$ has a white vertex of degree 2 except white vertex of degree 3.
\par\noindent
\underline{The second case.} Polynomials $p_\lambda$ have critical points only
of multiplicity 2. Thus vertices $a$ and $b$ generate critical points $a_1,a_2$
and $b_1,b_2$, respectively. Moreover, $p_\lambda(a_1)=p_\lambda(b_1)$,
$p_\lambda(a_2)=p_\lambda(b_2)$ and $p_\lambda(b_1)\neq p_\lambda(b_2)$. Let
the fifth critical point be $c=c_\lambda$. The vertex of $T_{12}$ of degree 3
cannot be generated by junction of points $a_1$ and $b_1$ (or $a_2$ and $b_2$),
because otherwise during the change of parameter $\lambda$ from 1 to 0 the
vertex of degree 3 of $T_{12}$ generates two critical points with same values.
Also, this vertex cannot be generated by junction of points $c$ and $a_1$ (for
example), because then $T_{12}$ has a vertex of degree 3 and a vertex of degree
2 of the same color. \end{proof}

All other pairs of trees are Z-homotopic. The construction of corresponding
Z-homotopy usually is quite straightforward. Let us describe some interesting
cases.
\begin{itemize}
    \item Tree $T_4$ and tree $T_{12}$. Let degree 2 vertices of $T_4$ be in
    points $\pm 1$, its degree 4 vertex --- in origin, degree 3 vertex of
    $T_{12}$ --- in origin and its degree 2 vertices --- in cubic roots of 1.

    Let us consider the polynomial $p=\int x^2(x-1)(x-a)(x-b)\,dx$ with
    condition $p(a)=p(b)$. Then $p$ is a Zolotarev polynomial with passport
    $\langle 3\,|\,2,2\,|\,2\rangle$. If $a=0$ and $b=-1$, then $p$ degenerates
    into Chebyshev polynomial that corresponds to the tree $T_4$. The change of
    parameter $a$ from $0$ to $-i$, to $2-i$, to $2$, to $2+\sqrt 3\,i/2$ and
    to $(-1+\sqrt 3\,i)/2$ induces the change of the parameter $b$
    from $-1$ to $(-1-\sqrt 3\,i)/2$.
    \item Tree $T_{10}$ and tree $T_{13}$. Let degree 3 vertex of $T_{10}$ be
    in origin, its degree 2 vertices --- in points $1$, $a_1\approx 1.57-0.03\,i$
    and $b_1\approx -0.57+0.58\,i$, degree 2 vertices of $T_{13}$ --- in points
    $0$, $\pm 1$ ¨ $\pm\sqrt{3}$.

    Let us consider the polynomial $p=\int x(x-1)(x-a)(x-b)(x-c)\,dx$ with
    conditions $p(a)=0$ and $p(b)=p(c)$. Then $p$ is a Zolotarev polynomial with
    passport $\langle 2,2\,|\,2,2\,|\,2\rangle$. If $a=a_1$, $b=b_1$ and $c=0$,
    then $p$ degenerates into Chebyshev polynomial that corresponds to the tree
    $T_{10}$. The change of parameter $b$ from $b_1$ to $-1$ induces the
    change of the parameter $a$ from $a_1$ to $\sqrt 3$ and the change of the
    parameter $c$ from $0$ to $-\sqrt 3$ (here $c$ moves along the arc in the lower
    half plane).
    \item Tree $T_{12}$ and tree $T_{13}$. Let degree 3 vertex of $T_{12}$ be
    in the point $i/\sqrt 3$, its degree 2 vertices --- in points $\pm 1$ and
    $\sqrt 3\,i$,  degree 2 vertices of $T_{13}$ --- in points
    $0$, $\pm 1$ ¨ $\pm 1/\sqrt{3}$.

    Let us consider the polynomial $p=\int (x^2-1)(x-a)(x-b)(x-c)\,dx$ with
    conditions $p(-1)=p(1)=p(c)$. Then $p$ is a Zolotarev polynomial with
    passport $\langle 2,2,2\,|\,2\,|\,2\rangle$. If $a=b=i/\sqrt 3$ and
    $c=\sqrt 3\,i$, then $p$ degenerates into Chebyshev polynomial that
    corresponds to the tree $T_{12}$. The change of parameter $a$ from
    $i/\sqrt 3$ to $1/\sqrt 3$ induces the change of the parameter $b$ from
    $i/\sqrt 3$ to $-1/\sqrt 3$ and the change of the parameter $c$ from
    $\sqrt 3\,i$ to $0$.
\end{itemize}

\section{Trees with seven edges}

Zolotarev polynomials of degree 7 with passport $\langle
2,2\,|\,2,2\,|\,2,2\rangle$ give a nontrivial example of absence of Z-homotopy
(nontrivial in the sense, that this absence cannot be explained by Lemma 2 or
Theorem 1). Without loss of generality we can assume, that the first critical
value is 0 and that corresponding critical points are 0 and 1. Then such
polynomial is of the form
$$p(x)=\int x(x-1)(x-a)(x-b)(x-c)(x-d)\,dx,$$  where
$$p(1)=0,\quad p(a)=p(b),\quad p(c)=p(d).$$ Algebraic variety $C$ in
4-dimensional space with coordinates $a,b,c,d$, defined by these conditions, is
reducible: it is the union of two components $C=C_1\cup C_2$ of degrees 8 and
16, respectively. Trees (up to mirror symmetry), that correspond to Zolotarev
polynomials from the first component, can be seen in the picture below:
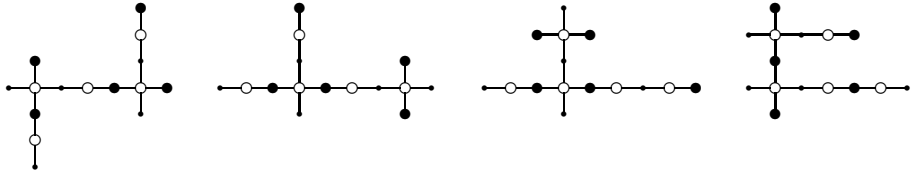
\[\begin{picture}(340,70) \put(0,35){\circle*{2}} \put(10,5){\circle*{2}}
\put(10,15){\circle{4}} \put(10,25){\circle*{4}} \put(10,35){\circle{4}}
\put(10,45){\circle*{4}} \put(20,35){\circle*{2}} \put(30,35){\circle{4}}
\put(40,35){\circle*{4}} \put(50,25){\circle*{2}} \put(50,35){\circle{4}}
\put(50,45){\circle*{2}} \put(50,55){\circle{4}} \put(50,65){\circle*{4}}
\put(60,35){\circle*{4}} \put(0,35){\line(1,0){8}} \put(10,5){\line(0,1){8}}
\put(10,17){\line(0,1){16}} \put(10,37){\line(0,1){8}}
\put(12,35){\line(1,0){16}} \put(32,35){\line(1,0){16}}
\put(50,25){\line(0,1){8}} \put(50,37){\line(0,1){16}}
\put(50,57){\line(0,1){8}} \put(52,35){\line(1,0){8}}

\put(80,35){\circle*{2}} \put(90,35){\circle{4}} \put(100,35){\circle*{4}}
\put(110,35){\circle{4}} \put(110,25){\circle*{2}} \put(110,45){\circle*{2}}
\put(110,55){\circle{4}} \put(110,65){\circle*{4}} \put(120,35){\circle*{4}}
\put(130,35){\circle{4}} \put(140,35){\circle*{2}} \put(150,35){\circle{4}}
\put(150,25){\circle*{4}} \put(150,45){\circle*{4}} \put(160,35){\circle*{2}}
\put(80,35){\line(1,0){8}} \put(92,35){\line(1,0){16}}
\put(110,25){\line(0,1){8}} \put(110,37){\line(0,1){16}}
\put(110,57){\line(0,1){8}} \put(112,35){\line(1,0){16}}
\put(132,35){\line(1,0){16}} \put(150,25){\line(0,1){8}}
\put(150,37){\line(0,1){8}} \put(152,35){\line(1,0){8}}

\put(180,35){\circle*{2}} \put(190,35){\circle{4}} \put(200,35){\circle*{4}}
\put(210,35){\circle{4}} \put(210,25){\circle*{2}} \put(210,45){\circle*{2}}
\put(210,55){\circle{4}} \put(210,65){\circle*{2}} \put(200,55){\circle*{4}}
\put(220,55){\circle*{4}} \put(220,35){\circle*{4}} \put(230,35){\circle{4}}
\put(240,35){\circle*{2}} \put(250,35){\circle{4}} \put(260,35){\circle*{4}}
\put(180,35){\line(1,0){8}} \put(192,35){\line(1,0){16}}
\put(210,25){\line(0,1){8}} \put(210,37){\line(0,1){16}}
\put(210,57){\line(0,1){8}} \put(200,55){\line(1,0){8}}
\put(212,55){\line(1,0){8}} \put(212,35){\line(1,0){16}}
\put(232,35){\line(1,0){16}} \put(252,35){\line(1,0){8}}

\put(280,35){\circle*{2}} \put(280,55){\circle*{2}} \put(290,25){\circle*{4}}
\put(290,35){\circle{4}} \put(290,45){\circle*{4}} \put(290,55){\circle{4}}
\put(290,65){\circle*{4}} \put(300,35){\circle*{2}} \put(300,55){\circle*{2}}
\put(310,35){\circle{4}} \put(310,55){\circle{4}} \put(320,35){\circle*{4}}
\put(320,55){\circle*{4}} \put(330,35){\circle{4}} \put(340,35){\circle*{2}}
\put(280,35){\line(1,0){8}} \put(280,55){\line(1,0){8}}
\put(290,25){\line(0,1){8}} \put(290,37){\line(0,1){16}}
\put(290,57){\line(0,1){8}} \put(292,35){\line(1,0){16}}
\put(292,55){\line(1,0){16}} \put(312,35){\line(1,0){16}}
\put(312,55){\line(1,0){8}} \put(332,35){\line(1,0){8}} \end{picture}\] The
order of monodromy group of Zolotarev polynomials from $C_1$ is 168.

Trees (up to mirror symmetry), that correspond to Zolotarev polynomials from
the second component, can be seen in the picture below:
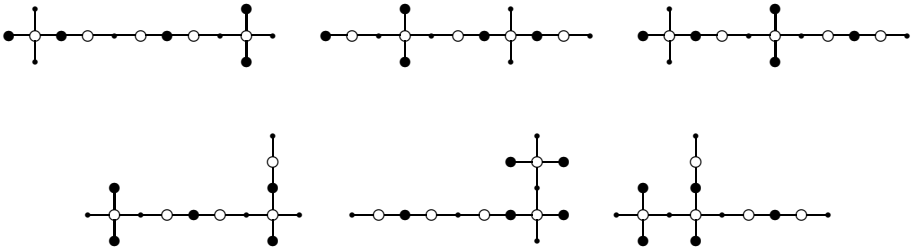
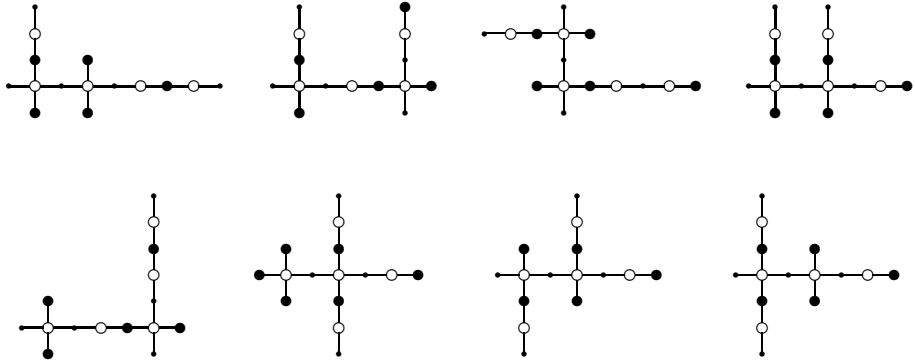
\[\begin{picture}(340,30) \put(0,15){\circle*{4}} \put(10,15){\circle{4}}
\put(10,5){\circle*{2}} \put(10,25){\circle*{2}} \put(20,15){\circle*{4}}
\put(30,15){\circle{4}} \put(40,15){\circle*{2}} \put(50,15){\circle{4}}
\put(60,15){\circle*{4}} \put(70,15){\circle{4}} \put(80,15){\circle*{2}}
\put(90,15){\circle{4}} \put(90,5){\circle*{4}} \put(90,25){\circle*{4}}
\put(100,15){\circle*{2}} \put(0,15){\line(1,0){8}} \put(10,5){\line(0,1){8}}
\put(10,17){\line(0,1){8}} \put(12,15){\line(1,0){16}}
\put(32,15){\line(1,0){16}} \put(52,15){\line(1,0){16}}
\put(72,15){\line(1,0){16}} \put(90,5){\line(0,1){8}}
\put(90,17){\line(0,1){8}} \put(92,15){\line(1,0){8}}

\put(120,15){\circle*{4}} \put(130,15){\circle{4}} \put(140,15){\circle*{2}}
\put(150,15){\circle{4}} \put(150,5){\circle*{4}} \put(150,25){\circle*{4}}
\put(160,15){\circle*{2}} \put(170,15){\circle{4}} \put(180,15){\circle*{4}}
\put(190,15){\circle{4}} \put(190,5){\circle*{2}} \put(190,25){\circle*{2}}
\put(200,15){\circle*{4}} \put(210,15){\circle{4}} \put(220,15){\circle*{2}}
\put(120,15){\line(1,0){8}} \put(132,15){\line(1,0){16}}
\put(150,5){\line(0,1){8}} \put(150,17){\line(0,1){8}}
\put(152,15){\line(1,0){16}} \put(172,15){\line(1,0){16}}
\put(190,5){\line(0,1){8}} \put(190,17){\line(0,1){8}}
\put(192,15){\line(1,0){16}} \put(212,15){\line(1,0){8}}

\put(240,15){\circle*{4}} \put(250,15){\circle{4}} \put(250,5){\circle*{2}}
\put(250,25){\circle*{2}} \put(260,15){\circle*{4}} \put(270,15){\circle{4}}
\put(280,15){\circle*{2}} \put(290,15){\circle{4}} \put(290,5){\circle*{4}}
\put(290,25){\circle*{4}} \put(300,15){\circle*{2}} \put(310,15){\circle{4}}
\put(320,15){\circle*{4}} \put(330,15){\circle{4}} \put(340,15){\circle*{2}}
\put(240,15){\line(1,0){8}} \put(250,5){\line(0,1){8}}
\put(250,17){\line(0,1){8}} \put(252,15){\line(1,0){16}}
\put(272,15){\line(1,0){16}} \put(292,15){\line(1,0){16}}
\put(312,15){\line(1,0){16}} \put(290,5){\line(0,1){8}}
\put(290,17){\line(0,1){8}} \put(332,15){\line(1,0){8}} \end{picture}\]

\[\begin{picture}(280,50) \put(0,15){\circle*{2}} \put(10,15){\circle{4}}
\put(10,5){\circle*{4}} \put(10,25){\circle*{4}} \put(20,15){\circle*{2}}
\put(30,15){\circle{4}} \put(40,15){\circle*{4}} \put(50,15){\circle{4}}
\put(60,15){\circle*{2}} \put(70,15){\circle{4}} \put(70,5){\circle*{4}}
\put(70,25){\circle*{4}} \put(70,35){\circle{4}} \put(70,45){\circle*{2}}
\put(80,15){\circle*{2}} \put(0,15){\line(1,0){8}} \put(10,5){\line(0,1){8}}
\put(10,17){\line(0,1){8}} \put(12,15){\line(1,0){16}}
\put(32,15){\line(1,0){16}} \put(52,15){\line(1,0){16}}
\put(70,5){\line(0,1){8}} \put(70,17){\line(0,1){16}}
\put(70,37){\line(0,1){8}} \put(72,15){\line(1,0){8}}

\put(100,15){\circle*{2}}  \put(110,15){\circle{4}} \put(120,15){\circle*{4}}
\put(130,15){\circle{4}} \put(140,15){\circle*{2}} \put(150,15){\circle{4}}
\put(160,15){\circle*{4}} \put(170,15){\circle{4}} \put(170,5){\circle*{2}}
\put(180,15){\circle*{4}} \put(170,25){\circle*{2}} \put(170,35){\circle{4}}
\put(160,35){\circle*{4}} \put(170,45){\circle*{2}} \put(180,35){\circle*{4}}
\put(100,15){\line(1,0){8}} \put(112,15){\line(1,0){16}}
\put(132,15){\line(1,0){16}} \put(152,15){\line(1,0){16}}
\put(170,5){\line(0,1){8}} \put(172,15){\line(1,0){8}}
\put(170,17){\line(0,1){16}} \put(160,35){\line(1,0){8}}
\put(170,37){\line(0,1){8}} \put(172,35){\line(1,0){8}}

\put(200,15){\circle*{2}} \put(210,15){\circle{4}} \put(210,5){\circle*{4}}
\put(210,25){\circle*{4}} \put(220,15){\circle*{2}} \put(230,15){\circle{4}}
\put(230,5){\circle*{4}} \put(230,25){\circle*{4}} \put(230,35){\circle{4}}
\put(230,45){\circle*{2}} \put(240,15){\circle*{2}} \put(250,15){\circle{4}}
\put(260,15){\circle*{4}} \put(270,15){\circle{4}} \put(280,15){\circle*{2}}
\put(200,15){\line(1,0){8}} \put(210,5){\line(0,1){8}}
\put(210,17){\line(0,1){8}} \put(212,15){\line(1,0){16}}
\put(230,5){\line(0,1){8}} \put(230,17){\line(0,1){16}}
\put(230,37){\line(0,1){8}} \put(232,15){\line(1,0){16}}
\put(252,15){\line(1,0){16}} \put(272,15){\line(1,0){8}}  \end{picture}\]

\[\begin{picture}(340,50)\put(0,15){\circle*{2}} \put(10,15){\circle{4}}
\put(10,5){\circle*{4}} \put(10,25){\circle*{4}} \put(10,35){\circle{4}}
\put(10,45){\circle*{2}} \put(20,15){\circle*{2}} \put(30,15){\circle{4}}
\put(30,5){\circle*{4}} \put(30,25){\circle*{4}} \put(40,15){\circle*{2}}
\put(50,15){\circle{4}} \put(60,15){\circle*{4}} \put(70,15){\circle{4}}
\put(80,15){\circle*{2}} \put(0,15){\line(1,0){8}} \put(30,5){\line(0,1){8}}
\put(30,17){\line(0,1){8}} \put(12,15){\line(1,0){16}}
\put(10,5){\line(0,1){8}} \put(10,17){\line(0,1){16}}
\put(10,37){\line(0,1){8}} \put(32,15){\line(1,0){16}}
\put(52,15){\line(1,0){16}} \put(72,15){\line(1,0){8}}

\put(100,15){\circle*{2}} \put(110,15){\circle{4}} \put(110,5){\circle*{4}}
\put(110,25){\circle*{4}} \put(110,35){\circle{4}} \put(110,45){\circle*{2}}
\put(120,15){\circle*{2}} \put(130,15){\circle{4}} \put(140,15){\circle*{4}}
\put(150,15){\circle{4}} \put(150,5){\circle*{2}} \put(150,25){\circle*{2}}
\put(150,35){\circle{4}} \put(150,45){\circle*{4}} \put(160,15){\circle*{4}}
\put(100,15){\line(1,0){8}} \put(110,5){\line(0,1){8}}
\put(110,17){\line(0,1){16}} \put(110,37){\line(0,1){8}}
\put(112,15){\line(1,0){16}} \put(132,15){\line(1,0){16}}
\put(150,5){\line(0,1){8}} \put(150,17){\line(0,1){16}}
\put(150,37){\line(0,1){8}} \put(152,15){\line(1,0){8}}

\put(200,15){\circle*{4}} \put(180,35){\circle*{2}} \put(190,35){\circle{4}}
\put(200,35){\circle*{4}} \put(210,5){\circle*{2}} \put(210,15){\circle{4}}
\put(210,25){\circle*{2}} \put(210,35){\circle{4}} \put(210,45){\circle*{2}}
\put(220,15){\circle*{4}} \put(220,35){\circle*{4}} \put(230,15){\circle{4}}
\put(240,15){\circle*{2}} \put(250,15){\circle{4}} \put(260,15){\circle*{4}}
\put(180,35){\line(1,0){8}} \put(192,35){\line(1,0){16}}
\put(200,15){\line(1,0){8}} \put(210,5){\line(0,1){8}}
\put(210,17){\line(0,1){16}} \put(210,37){\line(0,1){8}}
\put(212,35){\line(1,0){8}} \put(212,15){\line(1,0){16}}
\put(232,15){\line(1,0){16}} \put(252,15){\line(1,0){8}}

\put(280,15){\circle*{2}} \put(290,5){\circle*{4}} \put(290,15){\circle{4}}
\put(290,25){\circle*{4}} \put(290,35){\circle{4}} \put(290,45){\circle*{2}}
\put(300,15){\circle*{2}} \put(310,5){\circle*{4}} \put(310,15){\circle{4}}
\put(310,25){\circle*{4}} \put(310,35){\circle{4}} \put(310,45){\circle*{2}}
\put(320,15){\circle*{2}} \put(330,15){\circle{4}} \put(340,15){\circle*{4}}
\put(280,15){\line(1,0){8}} \put(290,5){\line(0,1){8}}
\put(290,17){\line(0,1){16}} \put(290,37){\line(0,1){8}}
\put(292,15){\line(1,0){16}} \put(310,5){\line(0,1){8}}
\put(310,17){\line(0,1){16}} \put(310,37){\line(0,1){8}}
\put(312,15){\line(1,0){16}} \put(332,15){\line(1,0){8}}
\end{picture}\]

\[\begin{picture}(330,70) \put(0,15){\circle*{2}} \put(10,15){\circle{4}}
\put(10,5){\circle*{4}} \put(10,25){\circle*{4}} \put(20,15){\circle*{2}}
\put(30,15){\circle{4}} \put(40,15){\circle*{4}} \put(50,15){\circle{4}}
\put(50,5){\circle*{2}} \put(60,15){\circle*{4}} \put(50,25){\circle*{2}}
\put(50,35){\circle{4}} \put(50,45){\circle*{4}} \put(50,55){\circle{4}}
\put(50,65){\circle*{2}} \put(0,15){\line(1,0){8}} \put(10,5){\line(0,1){8}}
\put(10,17){\line(0,1){8}} \put(12,15){\line(1,0){16}}
\put(32,15){\line(1,0){16}} \put(50,5){\line(0,1){8}}
\put(52,15){\line(1,0){8}} \put(50,17){\line(0,1){16}}
\put(50,37){\line(0,1){16}} \put(50,57){\line(0,1){8}}

\put(90,35){\circle*{4}} \put(100,35){\circle{4}} \put(100,25){\circle*{4}}
\put(100,45){\circle*{4}} \put(110,35){\circle*{2}} \put(120,35){\circle{4}}
\put(120,25){\circle*{4}} \put(120,15){\circle{4}} \put(120,5){\circle*{2}}
\put(120,45){\circle*{4}} \put(120,55){\circle{4}} \put(120,65){\circle*{2}}
\put(130,35){\circle*{2}} \put(140,35){\circle{4}} \put(150,35){\circle*{4}}
\put(90,35){\line(1,0){8}} \put(100,25){\line(0,1){8}}
\put(100,37){\line(0,1){8}} \put(102,35){\line(1,0){16}}
\put(120,5){\line(0,1){8}} \put(120,17){\line(0,1){16}}
\put(120,37){\line(0,1){16}} \put(120,57){\line(0,1){8}}
\put(122,35){\line(1,0){16}} \put(142,35){\line(1,0){8}}

\put(180,35){\circle*{2}} \put(190,35){\circle{4}} \put(190,5){\circle*{2}}
\put(190,15){\circle{4}} \put(190,25){\circle*{4}} \put(190,45){\circle*{4}}
\put(200,35){\circle*{2}} \put(210,35){\circle{4}} \put(210,25){\circle*{4}}
\put(210,45){\circle*{4}} \put(210,55){\circle{4}} \put(210,65){\circle*{2}}
\put(220,35){\circle*{2}} \put(230,35){\circle{4}} \put(240,35){\circle*{4}}
\put(180,35){\line(1,0){8}} \put(190,5){\line(0,1){8}}
\put(190,17){\line(0,1){16}} \put(190,37){\line(0,1){8}}
\put(192,35){\line(1,0){16}} \put(210,25){\line(0,1){8}}
\put(210,37){\line(0,1){16}} \put(210,57){\line(0,1){8}}
\put(212,35){\line(1,0){16}} \put(232,35){\line(1,0){8}}

\put(270,35){\circle*{2}} \put(280,5){\circle*{2}} \put(280,15){\circle{4}}
\put(280,25){\circle*{4}} \put(280,35){\circle{4}} \put(280,45){\circle*{4}}
\put(280,55){\circle{4}} \put(280,65){\circle*{2}} \put(290,35){\circle*{2}}
\put(300,35){\circle{4}} \put(300,25){\circle*{4}} \put(300,45){\circle*{4}}
\put(310,35){\circle*{2}} \put(320,35){\circle{4}} \put(330,35){\circle*{4}}
\put(270,35){\line(1,0){8}} \put(280,5){\line(0,1){8}}
\put(280,17){\line(0,1){16}} \put(280,37){\line(0,1){16}}
\put(280,57){\line(0,1){8}} \put(282,35){\line(1,0){16}}
\put(300,25){\line(0,1){8}} \put(300,37){\line(0,1){8}}
\put(302,35){\line(1,0){16}} \put(322,35){\line(1,0){8}} \end{picture}\] The
order of monodromy group of Zolotarev polynomials from $C_2$ is 2520.

The intersection $C_1\cap C_2$ consists of Chebyshev polynomials that
correspond to trees
\[\begin{picture}(150,20) \put(0,5){\circle*{2}} \put(10,5){\circle*{2}}
\put(10,15){\circle*{2}} \put(20,5){\circle*{2}} \put(30,5){\circle*{2}}
\put(30,15){\circle*{2}} \put(40,5){\circle*{2}} \put(50,5){\circle*{2}}
\put(0,5){\line(1,0){50}} \put(10,5){\line(0,1){10}} \put(30,5){\line(0,1){10}}

\put(70,3){or}

\put(100,5){\circle*{2}} \put(110,5){\circle*{2}} \put(120,15){\circle*{2}}
\put(120,5){\circle*{2}} \put(130,5){\circle*{2}} \put(140,15){\circle*{2}}
\put(140,5){\circle*{2}} \put(150,5){\circle*{2}} \put(100,5){\line(1,0){50}}
\put(120,5){\line(0,1){10}} \put(140,5){\line(0,1){10}} \end{picture}\]

However, the component $C_1$ contains Chebyshev polynomials that correspond to
trees
\[\begin{picture}(260,20) \put(0,5){\circle*{2}} \put(10,5){\circle*{2}}
\put(10,15){\circle*{2}} \put(20,5){\circle*{2}} \put(30,5){\circle*{2}}
\put(30,15){\circle*{2}} \put(40,5){\circle*{2}} \put(50,5){\circle*{2}}
\put(0,5){\line(1,0){50}} \put(10,5){\line(0,1){10}} \put(30,5){\line(0,1){10}}

\put(70,5){\circle*{2}} \put(80,5){\circle*{2}} \put(90,15){\circle*{2}}
\put(90,5){\circle*{2}} \put(100,5){\circle*{2}} \put(110,15){\circle*{2}}
\put(110,5){\circle*{2}} \put(120,5){\circle*{2}} \put(70,5){\line(1,0){50}}
\put(90,5){\line(0,1){10}} \put(110,5){\line(0,1){10}}

\put(140,5){\circle*{2}} \put(150,5){\circle*{2}} \put(150,15){\circle*{2}}
\put(160,5){\circle*{2}} \put(170,5){\circle*{2}} \put(170,15){\circle*{2}}
\put(180,5){\circle*{2}} \put(190,5){\circle*{2}} \put(140,5){\line(1,0){50}}
\put(160,5){\line(1,1){10}} \put(160,5){\line(-1,1){10}}

\put(210,5){\circle*{2}} \put(220,5){\circle*{2}} \put(230,15){\circle*{2}}
\put(230,5){\circle*{2}} \put(240,5){\circle*{2}} \put(250,15){\circle*{2}}
\put(250,5){\circle*{2}} \put(260,5){\circle*{2}} \put(210,5){\line(1,0){50}}
\put(240,5){\line(1,1){10}} \put(240,5){\line(-1,1){10}}
\end{picture}\] and the component $C_2$ contains Chebyshev polynomials, that
correspond to trees
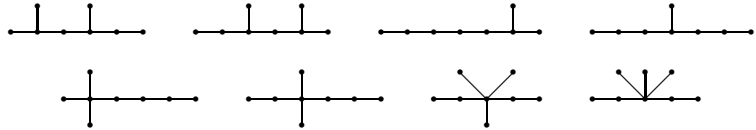
\[\begin{picture}(280,20) \put(0,5){\circle*{2}} \put(10,5){\circle*{2}}
\put(10,15){\circle*{2}} \put(20,5){\circle*{2}} \put(30,5){\circle*{2}}
\put(30,15){\circle*{2}} \put(40,5){\circle*{2}} \put(50,5){\circle*{2}}
\put(0,5){\line(1,0){50}} \put(10,5){\line(0,1){10}} \put(30,5){\line(0,1){10}}

\put(70,5){\circle*{2}} \put(80,5){\circle*{2}} \put(90,15){\circle*{2}}
\put(90,5){\circle*{2}} \put(100,5){\circle*{2}} \put(110,15){\circle*{2}}
\put(110,5){\circle*{2}} \put(120,5){\circle*{2}} \put(70,5){\line(1,0){50}}
\put(90,5){\line(0,1){10}} \put(110,5){\line(0,1){10}}

\put(140,5){\circle*{2}} \put(150,5){\circle*{2}} \put(160,5){\circle*{2}}
\put(170,5){\circle*{2}} \put(180,5){\circle*{2}} \put(190,5){\circle*{2}}
\put(200,5){\circle*{2}} \put(190,15){\circle*{2}} \put(140,5){\line(1,0){60}}
\put(190,5){\line(0,1){10}}

\put(220,5){\circle*{2}} \put(230,5){\circle*{2}} \put(240,5){\circle*{2}}
\put(250,5){\circle*{2}} \put(260,5){\circle*{2}} \put(270,5){\circle*{2}}
\put(280,5){\circle*{2}} \put(250,15){\circle*{2}} \put(220,5){\line(1,0){60}}
\put(250,5){\line(0,1){10}} \end{picture}\]
\[\begin{picture}(240,30) \put(0,15){\circle*{2}} \put(10,15){\circle*{2}}
\put(20,15){\circle*{2}} \put(30,15){\circle*{2}} \put(40,15){\circle*{2}}
\put(50,15){\circle*{2}} \put(10,5){\circle*{2}} \put(10,25){\circle*{2}}
\put(0,15){\line(1,0){50}} \put(10,5){\line(0,1){20}}

\put(70,15){\circle*{2}} \put(80,15){\circle*{2}} \put(90,15){\circle*{2}}
\put(100,15){\circle*{2}} \put(110,15){\circle*{2}} \put(120,15){\circle*{2}}
\put(90,5){\circle*{2}} \put(90,25){\circle*{2}} \put(70,15){\line(1,0){50}}
\put(90,5){\line(0,1){20}}

\put(140,15){\circle*{2}} \put(150,15){\circle*{2}} \put(160,15){\circle*{2}}
\put(170,15){\circle*{2}} \put(180,15){\circle*{2}} \put(160,5){\circle*{2}}
\put(150,25){\circle*{2}} \put(170,25){\circle*{2}}
\put(140,15){\line(1,0){40}} \put(160,15){\line(1,1){10}}
\put(160,15){\line(-1,1){10}} \put(160,15){\line(0,-1){10}}

\put(200,15){\circle*{2}} \put(210,15){\circle*{2}} \put(220,15){\circle*{2}}
\put(230,15){\circle*{2}} \put(240,15){\circle*{2}} \put(220,25){\circle*{2}}
\put(210,25){\circle*{2}} \put(230,25){\circle*{2}}
\put(200,15){\line(1,0){40}} \put(220,15){\line(1,1){10}}
\put(220,15){\line(-1,1){10}} \put(220,15){\line(0,1){10}} \end{picture}\] Thus
we see that trees
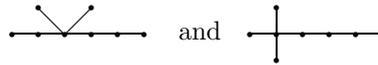
\[\begin{picture}(140,30) \put(0,15){\circle*{2}} \put(10,15){\circle*{2}}
\put(20,15){\circle*{2}} \put(30,15){\circle*{2}} \put(40,15){\circle*{2}}
\put(50,15){\circle*{2}} \put(10,25){\circle*{2}} \put(30,25){\circle*{2}}
\put(0,15){\line(1,0){50}} \put(20,15){\line(1,1){10}}
\put(20,15){\line(-1,1){10}} \put(63,13){and}

\put(90,15){\circle*{2}} \put(100,15){\circle*{2}} \put(110,15){\circle*{2}}
\put(120,15){\circle*{2}} \put(130,15){\circle*{2}} \put(140,15){\circle*{2}}
\put(100,5){\circle*{2}} \put(100,25){\circle*{2}} \put(90,15){\line(1,0){50}}
\put(100,5){\line(0,1){20}} \end{picture}\] for example, are not Z-homotopic in
the class of Zolotarev polynomials with the passport $\langle
2,2\,|\,2,2\,|\,2,2\rangle$ (although they are Z-homotopic in the class with
the passport $\langle 4\,|\,2\,|\,2\rangle$).

\end{document}